\documentclass[reqno,article,12pt,twosided,a4paper]{amsart}
\usepackage{amssymb,amsfonts,latexsym,mathrsfs}
\usepackage[all]{xy}
\usepackage{array}

\usepackage{amstext}
\usepackage{amssymb}
\usepackage{amsfonts}
\usepackage{amsthm}
\usepackage{amsopn}
\usepackage{amsbsy}
\usepackage{layout}

\newtheorem{theorem}{Theorem}[section]
\newtheorem{lemma}[theorem]{Lemma}
\newtheorem{proposition}[theorem]{Proposition}
\newtheorem{corollary}[theorem]{Corollary}

\theoremstyle{definition}
\newtheorem{definition}[theorem]{Definition}

\newtheorem{example}[theorem]{Example}

\newtheorem{conjecture}[theorem]{Conjecture}
\newtheorem{remark}[theorem]{Remark}

\newcommand{\h}{\mathbf{H}}

\renewcommand{\b}{\mathfrak{b}}

\newcommand{\Ga}{\Gamma}
\newcommand{\Z}{\mathbb{Z}}

\newcommand{\converges}{\Rightarrow}
\newcommand{\f}{\mathscr{F}}
\newcommand{\lh}{\mathbf{LH}}
\newcommand{\lhf}{\mathbf{LH}\f}
\newcommand{\hf}{\mathbf{H}\f}
\newcommand{\determines}{\converges}
\newcommand{\x}{\mathscr{X}}
\newcommand{\y}{\mathscr{Y}}
\newcommand{\injects}{\hookrightarrow}

\newcommand{\ben}{\begin{enumerate}}
\newcommand{\een}{\end{enumerate}}

\newcommand{\Q}{{\mathbb Q}}

\begin{document}

\title[A conjecture of Moore]{Nilpotency of Bocksteins, Kropholler's hierarchy and a conjecture of Moore}
\author{Eli Aljadeff and Ehud Meir}
\date{May 10, 2009}
\maketitle
\begin{abstract}
A conjecture of Moore claims that if $\Ga$ is a group and $H$ a finite index subgroup of $\Ga$ such that $\Gamma -  H$ has no elements of prime order (e.g. $\Gamma$ is torsion free), then a $\Gamma$-module which is projective over $H$ is projective over $\Gamma$. The conjecture is known for finite groups. In that case, it is a direct
consequence of Chouinard's theorem which is based on a fundamental result of Serre on the vanishing of products of Bockstein operators. It was observed by Benson, using a construction of Baumslag, Dyer and Heller, that the analog of Serre's Theorem for infinite groups is not true in general. We prove that the conjecture is true for groups which satisfy the analog of Serre's theorem. Using a result of Benson and Goodearl, we prove that the conjecture holds for all groups inside Kropholler's hierarchy $\lhf$, extending a result of Aljadeff, Cornick, Ginosar, and Kropholler. We show two closure properties for the class of pairs of groups $(\Ga,H)$ which satisfy the conjecture, the one is closure under morphisms, and the other is a closure operation which comes from Kropholler's construction. We use this in order to exhibit cases in which the analog of Serre's theorem does not hold, and yet the conjecture is true.
We will show that in fact there are pairs of groups $(\Ga,H)$ in which $H$ is a perfect normal subgroup of prime index in $\Ga$,
and the conjecture is true for $(\Ga,H)$. Moreover, we will show that it is enough to prove the conjecture for groups of this kind only.
%We also show that in a certain sense the class of groups which satisfies the conjecture is closed under morphisms, and use this in order to exhibit some new cases in which the conjecture is true.
\end{abstract}

\begin{center}
\end{center}

\begin{section}{Introduction and statement of results}\label{introduction}
A well known result of Serre (see \cite{Swan}) says that if $\Gamma$ is a torsion
free group and $H$ a subgroup of finite index then they have the
same cohomological dimension. The main task in the proof is to
show that if $cd(H) < n$ then $cd(\Gamma) < \infty$. Once this is
achived it is not difficult to show that in fact $cd(\Gamma)=n$.

In 1976 J. Moore posed a conjecture which is a far reaching
generalization of Serre's theorem.
\begin{conjecture}
Let $\Gamma$ be a torsion free group and $H$ a subgroup of finite
index. Let $M$ be a $\mathbb{Z} \Gamma$ module. Then $M$ is
projective if (and only if) it is projective as a
$\mathbb{Z}H$-module. \end{conjecture}

More generally
\begin{conjecture} Let $\Gamma$ be any group and $H$ a subgroup of finite index.
Assume no elements of prime order lie in $\Gamma - H$
(cf. Moore's condition). Then the same conclusion holds, that is,
if $M$ is a $\mathbb{Z} \Gamma$-module which is projective over
$\mathbb{Z}H$, then is also projective over $\mathbb{Z} \Gamma$.
\end{conjecture}
Throughout the paper we will use the following terminology:
\begin{enumerate}
\item
We say that Moore's conjecture holds for the pair $(\Gamma,H)$ if
(a) $(\Gamma,H)$ satisfies Moore's condition. (b) Every $\mathbb{Z}
\Gamma$-module $M$ which is projective over $\mathbb{Z}H$, is also
projective over $\mathbb{Z} \Gamma$.
\item
We say that Moore's conjecture holds for a group $\Gamma$ if the
conjecture holds for any pair $(\Gamma, H)$ which satisfies Moore's
condition. So for instance Moore's conjecture holds for a cyclic
group of prime order $\Gamma$ but not for the pair $(\Gamma,
\{e\})$.
\end{enumerate}

$\mathbf{Remarks}$
\begin{enumerate}
\item
Moore's conjecture may be rephrased as follows: Let $\Gamma$ and
$H$ be as above, then for any $\mathbb{Z} \Gamma$-module $M$,
$Proj.dim_{\mathbb{Z} \Gamma}(M)=
Proj.dim_{\mathbb{\mathbb{Z}}H}(M)$ (indeed if the conjecture is
true and \\$Proj.dim_{\mathbb{\mathbb{Z}}H}(M)=n$, then the $n$-th
Syzygy of any projective resolution of $M$ over $\mathbb{Z}
\Gamma$ is projective).
\item
Moore's condition is necessary in the conjecture in the following
sense (See \cite {Al2}): Let $\Gamma$ be a group and $H$ a normal subgroup
of finite index. If there are elements of prime order in $\Gamma - H$
then there exists a $\mathbb{Z} \Gamma$-module $M$,
projective over $\mathbb{\mathbb{Z}}H$ but not projective over
$\mathbb{Z} \Gamma$.

\item
It is not difficult to show (as in Serre's theorem) (see \cite{Swan}), that
if $\Gamma$ is arbitrary and $H$ is of finite index then
$Proj.dim_{\Z \Gamma}(M) < \infty$ implies $Proj.dim_{\Z
\Gamma}(M)=Proj.dim_{\Z H}(M)$. So the real content of Moore's
conjecture is in the statement: If $\Gamma$ is torsion free then
$Proj.dim_{\Z H}(M) < \infty \Longrightarrow Proj.dim_{\Z \Gamma}(M) <
\infty$.

\item
Serre's theorem is the case where $M=\Z$ with the trivial $\Gamma$
action.
\end{enumerate}
%\end{remark}

The first step in our analysis will be to reduce the conjecture to
the case in which the group $H$ is normal in $\Gamma$ and of prime
index $p$ (see section \ref{general}). In our exposition below we will assume
that this is indeed the case unless we state otherwise.

Notice that the general Moore's conjecture is meaningful for
finite groups. In that case it is known to be true and it is in
fact a consequence of Chouinard's theorem (see \cite{Ch}). Chouinard's
proof is based on a fundamental idea of Serre which is formulated
below only in a special case. We start with a definition.

Let $\Gamma$ be any group and $H$ a normal subgroup of prime index
$p$. Consider the non-split extension $\widehat{\beta}:
1\longrightarrow \Z\longrightarrow \Z\longrightarrow \Gamma/H \cong
\Z_p \longrightarrow 1$ as an element in $H^{2}(\Gamma/H, \Z)$ and
let $\beta_{\Gamma/H}= inf_{\Gamma/H}^{\Gamma}(\widehat{\beta})
\in H^{2}(\Gamma, \Z)$ where $inf$ denotes the inflation map. We
refer to the element $\beta_{\Gamma/H}$ as the Bockstein operator
or just the Bockstein, that corresponds to the pair $(\Gamma,H)$.

\begin{theorem}\label{Serre}(Serre, see \cite{Evens})
Let $\Gamma$ be a finite group and $H$ a normal subgroup of prime
index. Then Moore's condition holds if and only if
$\beta_{\Gamma/H}$ is nilpotent in the cohomology ring
$H^{*}(\Gamma,\mathbb{Z})$.

\end{theorem}

Chouinard's idea was to use the nilpotency of the Bockstein
$\beta_{\Gamma/H}$ in order to get the desired result (it should
be mentioned that his result is more general).

Our first observation is that the nilpotency of $\beta_{\Gamma/H}$
plays the same role for infinite groups as for finite groups.

\begin{proposition}\label{nil-conj}

Let $(\Gamma,H)$ be a group and a normal subgroup of prime index. If
$\beta_{\Gamma/H}$ is nilpotent in the cohomology ring, then Moore's
conjecture is true for $(\Gamma,H)$ (in particular Moore's condition
holds for $(\Gamma,H)$).

\end{proposition}

Now suppose there is a subgroup $K$ of $H$ which is normal in
$\Gamma$ and such that $\beta_{\Gamma/K, H/K}$ is nilpotent in
$H^{*}(\Gamma/K, \mathbb{Z})$. Then, clearly, $\beta_{\Gamma/H}$ is
nilpotent in $H^{*}(\Gamma, \mathbb{Z})$ and hence Moore's
conjecture is true for $(\Gamma,H)$. Combining with \ref{nil-conj} and with \ref{Serre} we prove an equivalent version of
Theorem 1.2 of \cite{Al2}
%(Eli- I did not find the exact place in which this theorem appears at first. Can you please find a reference?)
\begin{corollary}\label{finite-K}
Let $(\Gamma, H,K)$ as above where $K$ is of finite index. If
Moore's condition holds for $(\Gamma/K, H/K)$ then Moore's conjecture
is true for $(\Gamma,H)$.
\end{corollary}
%\begin{proof}
%
%This follows at once from Serre and the Proposition.
%
%\end{proof}
For instance, this is the case for the Thompson's group $F$. This
is easily seen from the fact that $F/K \cong \mathbb{Z}^{2}$ where $K$ is the intersection of all finite index subgroups of $F$.

The next result is the key for the proof of Theorem \ref{main1} below. It
extends considerably Corollary \ref{finite-K} by replacing pairs of the form
$(\Gamma/K, H/K)$ where $\Gamma/K$ is finite by pairs
$(\Gamma^{'},H^{'})$ which satisfy Moore's conjecture.
The proof uses a result of Benson and Goodearl.

\begin{proposition}\label{closure}

Suppose that $(\Gamma,H)$ and $(\Gamma^{'},H^{'})$ are two pairs of
a group and a finite index normal subgroup. Suppose also that
$\Gamma/H\cong \Gamma^{'}/H^{'}$, and we have a morphism of short
exact sequences

\begin{displaymath}
\xymatrix{1\ar[r] & H\ar[d]\ar[r] & \Gamma\ar[d]^{\phi}\ar[r] &
\Gamma/H\ar[d]^{\cong}\ar[r] & 1\\ 1\ar[r] & H^{'}\ar[r] &
\Gamma^{'}\ar[r] & \Gamma^{'}/H^{'}\ar[r] & 1.}
\end{displaymath}

Suppose that $(\Gamma^{'},H^ {'})$ satisfies Moore's conjecture.
Then $(\Gamma,H)$ satisfies Moore's conjecture as well.
\end{proposition}
The proof of proposition \ref{closure} will be given in section \ref{closure section}

Unlike the case of finite groups, it was observed by D. Benson
(see \cite{Al1}) that the Bockstein that corresponds to a pair $(\Gamma,
H)$ may not be nilpotent even in cases where $\Gamma$ is torsion
free. Indeed, using a strengthened version (proved by Baumslag
Dyer and Heller) of a construction of Kan and Thurston (See
\cite{KT},\cite{BDH}) one can show the existence of a torsion free group
$\Gamma$, with a perfect subgroup $H$ of prime index, with
$\beta_{(\Gamma, H)}$ non-nilpotent.

One of our main points in this paper is to show that Moore's
conjecture may be true for such pairs $(\Gamma,H)$.

In fact in section \ref{main section} we will show that much more is true.

\begin{theorem}\label{main1}

For every pair of groups $(\Gamma_0,H_0)$ (as usual $H_0$ is normal and of
prime index in $\Gamma_0$) there exists a pair of groups
$(\Gamma,H)$ as follows:

\begin{enumerate}

\item

$H$ is normal and of prime index in $\Gamma$

\item

The group $H$ is acyclic and hence the Bockstein
$\beta_{(\Gamma,H)}$ is not nilpotent.

\item
There is an embedding $\phi:\Gamma_0 \longrightarrow
\Gamma$. Moreover, with this embedding $H_0= H\cap
\Gamma_0$ (i.e. $\phi$ induces an isomrphism of $\Gamma_0/H_0
\cong \Gamma/H$.

\item

The pair $(\Gamma,H)$ satisfies Moore's
condition if and only if the pair $(\Gamma_0,H_0)$ does.

\item

The pair $(\Gamma,H)$ satisfies Moore's
conjecture if and only if the pair $(\Gamma_0,H_0)$ does.

\end{enumerate}

\end{theorem}

Our second main result in this paper is concerning Kropholler's
hierarchy.

In \cite{K2} Kropholler constructed a rather large family of groups,
denoted by $\hf$, which contains for instance all groups with virtual
finite cohomological dimension, all finitely generated soluble
groups, and all linear groups. This family is subgroup closed and closed under
group extensions (See section \ref{Kropholler} or \cite{K2} for the precise definition
of $\hf$). The class $\lhf$ is the class of groups whose finitely
generated subgroups are in $\hf$. In \cite{ACGK} it was shown that Moore's
conjecture holds for all groups in $\lhf$ under the condition that
the module $M$ (in the formulation of the conjecture) is finitely
generated.
Here we prove that the above restriction on $M$ may be removed.

\begin{theorem}\label{LHF thm}
Moore's conjecture holds for any group in $\lhf$.
\end{theorem}

We conclude by combining Theorem \ref{main1} and Theorem \ref{LHF thm}

\begin{theorem}\label{still in LHF}
With the notation of Theorem \ref{main1}, the group $\Gamma$ is
in $\lhf$ if and only if $\Gamma_0$ is in $\lhf$.
\end{theorem}

\begin{corollary}
There exist pairs of groups $(\Gamma, H)$ such that $\Gamma$ is in
$\lhf$ (and hence $(\Ga,H)$ satisfies Moore's conjecture) and the corresponding
Bockstein $\beta_{\Gamma,H}$ is not nilpotent.
\end{corollary}

On the other hand, applying Theorem \ref{main1} to the Thompson's group we obtain

\begin{corollary}
There exist pairs of groups $(\Gamma, H)$ such that $\Gamma$ in not
in $\lhf$, $(\Gamma, H)$ satisfies Moore's conjecture and the
corresponding Bockstein $\beta_{\Gamma,H}$ is not nilpotent.
\end{corollary}

The paper is organized as follows: In section \ref{general} we present some
general results which are essential for the rest of the paper. In
particular we prove the reduction to normal subgroups of prime index. In section \ref{closure section}
we prove Theorem \ref{closure}. In section \ref{Kropholler} we recall Kropholler's
construction and prove Theorem \ref{LHF thm}. In section \ref{main section} we prove Theorem \ref{main1}.
\end{section}

\begin{section}{Reductions, generalities, and the case of finite groups}\label{general}
We start with the reduction mentioned in the introduction:
\begin{lemma}
Suppose that Moore's conjecture is true for every pair of groups
$(\Gamma,H)$ (which satisfies Moore's condition) such that $H$ is a normal subgroup of $\Gamma$ of prime
index.
%and Moore's condition holds for $(\Gamma,H)$.
Then Moore's conjecture is true for any pair of groups $(\Gamma,H)$ which satisfies Moore's condition.
\end{lemma}
\begin{proof}
Let $(\Gamma,H)$ be any pair of groups which satisfies Moore's
condition. We will show that under the assumption of the lemma,
the pair $(\Ga,H)$ satisfies Moore's conjecture.
Let $$\bar{H}=core(H)=\bigcap_{g\in \Gamma}gHg^{-1}.$$
Then $\bar{H}$ is a normal subgroup of $\Gamma$ of finite index.
It is easy to see that if $(\Gamma,H)$ satisfies Moore's
condition, the same is true for $(\Gamma,\bar{H})$. We thus assume
without loss of generality that $H$ is normal in $\Gamma$ (since
if the conjecture is true for $(\Gamma,\bar{H})$, it is clearly
true for $(\Gamma,H)$).
%In other words- it is enough to prove the
%conjecture in the case that $H$ is normal in $\Gamma$.

Let $M$ be a $\Gamma$-module which is projective over $H$. We would
like to prove that $M$ is projective over $\Gamma$. Consider the
finite group $F=\Gamma/H$. For every prime number $p_i$ which
divides the order $n$ of $F$, let $\bar{P_i}$ be a $p_i$-Sylow
subgroup of $F$, and let $P_i$ be its inverse image in $\Gamma$.
Since $P_i/H=\bar{P_i}$ is a finite $p_i$ group, there is a finite chain of
subgroups
$$H=H_0< H_1<\ldots<P_i$$ such that
$H_{j+1}/H_j\cong \mathbb{Z}_p$. By induction, it follows now
from our assumption that $M$ is projective over every $H_j$, and
therefore also over $P_i$ for every $i$. Let $N$ be any
$\Gamma$-module. The map $cor^{P_i}_{\Gamma}\cdot
res^{\Gamma}_{P_i} :Ext^1_{\Gamma}(M,N)\rightarrow
Ext^1_{\Gamma}(M,N)$ is multiplication by $|F/\bar{P_i}|$. Thus
all the numbers $|F/\bar{P_i}|$ annihilates $Ext^1_{\Gamma}(M,N)$.
These numbers are coprime (since $\bar{P_i}$ is a $p_i$-Sylow
subgroup of $F$), so $Ext^1_{\Gamma}(M,N)=0$. Since $N$ was arbitrary,
this means that $M$ is projective over $\Gamma$ as required.
\end{proof}

We assume therefore that $\Gamma/H$ is cyclic of prime order $p$.
Let us denote a generator of $\Gamma/H$ by $x$. We shall use some well known
results on the cohomology of finite cyclic groups. For a proof of
these, see for example the book of Mac Lane, \cite{Maclane}. We
know that the second cohomology group $H^2(\Gamma/H,\mathbb{Z})$
is cyclic of order $p$. A generator $\hat{\beta}$ of this group is
given by the exact sequence

\begin{displaymath} \label{bockstein}
\hat{\beta} : 1 \rightarrow \Z\rightarrow \Z\Gamma/H\rightarrow
\Z\Gamma/H\rightarrow \Z\rightarrow 1
\end{displaymath}

where the first map is given by the inclusion $1\mapsto
\sum_{i=0}^{p-1}x^i$, the second map is the $\Gamma/H$ map which
sends 1 to $1-x$, and the third map is the natural projection (the augmentation map).
Consider the inflation of $\hat{\beta}$ to $\Gamma$,
$\beta=inf^{\Gamma}_{\Gamma/H}(\hat{\beta})$. As an exact sequence,
$\beta$ is represented by the same exact sequence as $\hat{\beta}$ %\ref{bockstein},
where the modules are now considered as $\Gamma$-modules. Notice that as
an exact sequence of $\Z$-modules, $\beta$ splits, and therefore if
$M$ is any $\Gamma$-module, then $\beta\otimes_{\Z} M$ is also
exact, and we can consider $\beta_M=\beta\otimes_{\Z} M\in
Ext^2_{\Gamma}(M,M)$.
\begin{remark}
The description of $\beta$ given above is different from the one given in the introduction.
The fact that the definitions are equivalent follows easily from results which appear in \cite{Maclane}.
%We gave a different description of $\beta$ in the introduction.
%The fact that the definition from the introduction and the definition given
%ere are equivalent follows easily from results which appears on \cite{Maclane}.
\end{remark}
As we shall soon see, $\beta$ and $\beta_M$ play a decisive role in
what follows. Recall first that if $N$ is a $\Gamma/H$-module, then
for every natural number $i>0$ the cup product with
$\hat{\beta}$ defines an isomorphism $H^i(\Gamma/H,N)\rightarrow
H^{i+2}(\Gamma/H,N)$. If $M$ is also a $\Gamma/H$ module, then one
can define in the obvious way the cocycle $\hat{\beta}_M\in
Ext^2_{\Gamma/H}(M,M)$ such that $inf{\hat{\beta_M}}=\beta_M$, and
cup product with $\hat{\beta}_M$ defines an isomorphism
$Ext^i_{\Gamma/H}(M,N)\rightarrow Ext^{i+2}_{\Gamma/H}(M,N)$ for every
$i$.

The nilpotency of $\beta_M$ plays a similar role for infinite groups
as for finite groups

\begin{proposition} (see \cite{QV} and \cite{Carlson} for the finite case)

Let $\Gamma$ be a group, and let $H$ be a normal subgroup of prime
index such that $(\Gamma,H)$ satisfies Moore's condition. Let $M$
be a $\Gamma$-module. Then $\beta_M\in Ext^2_{\Gamma}(M,M)$ is
nilpotent if and only if it is true that any $\zeta\in
Ext^*_{\Gamma}(M,M)$ with nilpotent restriction to $H$ is
nilpotent.

\end{proposition}

\begin{proof}

One direction is easy- if every element $\zeta$ with nilpotent
restriction to $H$ is nilpotent, then in particular $\beta$, whose
restriction to $H$ is zero, is nilpotent. The proof of the other
direction is exactly the same as the proof in case the group $\Ga$
is finite, and uses the Quillen-Venkov lemma. The Quillen-Venkov lemma says
that if we consider the LHS spectral sequence
$$E^2_{s,t}=H^s(\Gamma/H,Ext^t_H(M,M))\converges
Ext^{s+t}_{\Gamma}(M,M),$$ then multiplication by $\beta_M$ defines
an epimorphism $E^r_{s,t}\rightarrow E^r_{s+2,t}$ which is also an
isomorphism in case $s\geq r$. Then it can easily be shown that if
$\zeta\in Ext^*_{\Gamma}(M,M)$ has nilpotent restriction to $H$,
then there exist an $i$ such that $\zeta^i$ is divisible by
$\beta_M$, and this proves the claim. The proof of Quillen-Venkov
lemma for the case of finite groups can also be found in
\cite{Evens}. Since the proof of the lemma does not really depend on
the finiteness of the group $\Gamma$, our claim follows.

\end{proof}

We can now prove Proposition \ref{nil-conj}. In fact we present a stronger
statement using $\beta_M$ rather than $\beta$.

\begin{proposition}\label{nilpotency}

Let $\Gamma$ and $H$ be as above. Let $M$ be a $\Gamma$-module which
is projective as an $H$-module. Then the following conditions are
equivalent: \\a. $M$ is projective over $\Gamma$. \\b. $\beta_M=0$ \\c.
$\beta_M$ is nilpotent.
\end{proposition}
\begin{proof}
That a $\determines$ b $\determines$ c is trivial, since if $M$ is
projective over $\Gamma$ then $Ext^2_{\Gamma}(M,M)=0$. Suppose that
$\beta_M$ is nilpotent. Let $N$ be any $\Gamma$-module. Consider the
LHS spectral sequence
$$E^2_{s,t}=H^s(\Gamma/H,Ext^t_H(M,N))\converges
Ext^{s+t}_{\Gamma}(M,N).$$ Since $M$ is projective over $H$, this
sequence collapses at the $E^2$ term, and thus
$Ext^n_{\Gamma}(M,N)=H^n(\Gamma/H,Hom_H(M,N))$. It follows that
multiplication by $\beta_M$ is an isomorphism
$Ext^n_{\Gamma}(M,N)\rightarrow Ext^{n+2}_{\Gamma}(M,N)$ for every
$n>0$. Since $\beta_M$ is nilpotent, it follows that
$Ext^n_{\Gamma}(M,N)=0$ for every $n>0$ and our claim
follows.\end{proof}

\begin{remark}

The proof above mimics the proof of Chouinard's theorem in
the case of finite groups.

\end{remark}

Using the proposition above, we can derive two easy corollaries:

\begin{corollary}

Suppose that $\Gamma$ and $H$ are as above, and that the Bockstein
element $\beta$ is nilpotent. Then $\beta_M$ is nilpotent as well
and therefore Moore's conjecture holds for
$(\Gamma,H)$.

\end{corollary}

\begin{proof}

It is easy to see that $\beta^i\otimes_{\Z}M = \beta_M^i$
for every $i$, and therefore if $\beta$ is nilpotent then $\beta_M$
is nilpotent as well.

\end{proof}

\begin{corollary}\label{finiteness}

Suppose that $\Gamma$ and $H$ are as above, and $M$ is a
$\Gamma$-module of finite projective dimension, which is projective
over $H$. Then $M$ is projective over $\Gamma$. In particular, if
$\Gamma$ is a group of finite cohomological dimension, then Moore's
conjecture is true for $(\Gamma,H)$.

\end{corollary}

\begin{proof}

If $M$ has finite projective dimension, then $\beta_M$ must be
nilpotent and proposition \ref{nilpotency} applies. If $\Gamma$ has finite
cohomological dimension, then any $M$ has finite projective
dimension.
\end{proof}
\begin{remark}\label{finiteness remark}
Corollary \ref{finiteness} is actually true in general.
That is, if $H$ is a finite index subgroup of $\Ga$, and
$M$ is a $\Ga$-module which is $H$-projective and which has
a finite length resolution over $\Ga$, then $M$ is projective over $\Ga$.
This can be proved using the Eckmann-Shapiro Lemma.
The proof is a variant of the proof of Lemma 9.1 of \cite{Swan}.
\end{remark}

As mentioned above, the Bockstein $\beta$ is nilpotent
whenever there is a finite index subgroup $K$ of $H$, normal in
$\Gamma$, such that $(\Gamma/K,H/K)$ satisfies Moore's condition.
One may ask if the existence of such finite quotient is necessary
for the nilpotency of $\beta_{\Gamma,H}$? The following two examples
shows that the answer is negative.

\begin{example}
Consider the group $G=\Z$, and let $p$ be a prime number. The group $G$ has a finite index normal subgroup $\tilde{H}=p\Z$.
The group $\tilde{H}$ is also embedded in the additive group of rational numbers $\Q$.
We can thus form the free product with amalgamation $\Ga=G*_{\tilde{H}}\Q$.
Since $\Q$ has no finite index subgroups, each finite index subgroup of $\Ga$ must contains $\Q$,
and therefore also $\tilde{H}$. It follows easily that the only finite index subgroup of $\Ga$ is the normal closure of $\Q$, which we shall denote by $H$.
This is a normal subgroup of index $p$. It is easy to see that there is no subgroup $K<H$ of finite index such that $(\Ga/K,H/K)$ satisfies Moore's condition.
However, $\Ga$ has finite cohomological dimension, being the free product with amalgamation of groups of finite cohomological dimension, and the Bockstein $\beta$ is therefore nilpotent.
\end{example}
\begin{example}
As another example, consider the group $SL_{n}(\mathbb{Z})$ for $n > 2$. Let $m>3$ be a natural number, and let
$\Gamma_{n}(m)=ker(\phi)$ where $\phi:
SL_{n}(\mathbb{Z})\longrightarrow SL_{n}(\mathbb{Z}_{m})$ is the natural map. The group $\Gamma_n(m)$ is known as a
congruence subgroup. Recall (see \cite{cong}, Lemma 4.7.11 and Proposition 4.7.12) that $\Ga_{n}(m)$ is torsion free, residually finite, and
moreover has finite cohomological dimension.
%Clearly this says that any Bockstein of $(\Gamma,H)$ is nilpotent.
On the other hand, it is known that there is torsion in the profinite
completion $\widehat{\Ga_n(m)}$ of $\Ga_{n}(m)$.
Let us denote an element of prime order $p$ in $\widehat{\Ga_n(m)}$ by $x$.
If $H$ is a finite index subgroup of $\Ga_n(m)$, then $\widehat{H}$,
the profinite completion of $H$, is a finite index subgroup of $\widehat{\Ga_n(m)}$.
It is easy to see that the intersection of all such $\widehat{H}$ is trivial
(since both $\Ga_n(m)$ and $\widehat{\Ga_n(m)}$ are residually finite),
and therefore there is a subgroup $H$ of $\Ga_n(m)$ such that $x\notin \widehat{H}$.
By taking the core of $H$ if necessary, we may assume that $H$ is normal in $\Ga_n(m)$.
Since elements of $\widehat{\Ga_n(m)}$ are just coherent families of elements of $\Ga_n(m)/K$,
where $K$ runs over all finite index normal subgroups of $\Ga_n(m)$,
it is reasonable to speak about the element $x$ $mod$ $H$ of $\Ga_n(m)/H$.
We can lift this element to an element $y$ of $\Ga_n(m)$ which satisfies $y\notin H$ and $y^p\in H$.
Consider now the subgroup $R$ of $\Ga_n(m)$ which is generated by $H$ and $y$.
The group $H$ is a normal subgroup of $R$ of prime index $p$. Since $\Ga_n(m)$ has finite cohomological dimension,
the same is true for $R$ and therefore the Bockstein of $(R,H)$ is nilpotent.
On the other hand, it is easy to see that if we consider $\widehat{R}$ as a subgroup of $\widehat{\Ga_n(m)}$,
then $x$ is in $\widehat{R} - \widehat{H}$. Therefore there cannot be a finite index normal subgroup $K$ of $R$
such that $(R/K,H/K)$ satisfies Moore's condition, even though the Bockstein is nilpotent.
\end{example}
\end{section}

\begin{section}{Closure under short exact sequences}\label{closure section}

In this section we present the proof of Proposition \ref{closure} from the introduction.
Recall that we have the following map of short exact sequences:
\begin{displaymath}
\xymatrix{1\ar[r] & H\ar[d]\ar[r] & \Gamma\ar[d]^{\phi}\ar[r] &
\Gamma/H\ar[d]^{\cong}\ar[r] & 1\\ 1\ar[r] & H^{'}\ar[r] &
\Gamma^{'}\ar[r] & \Gamma^{'}/H^{'}\ar[r] & 1.}
\end{displaymath}

We know that $(\Ga',H')$ satisfies Moore's conjecture, and we would like to prove that $(\Ga,H)$ satisfies Moore's conjecture.
In the course of the proof, we will use the following result of Benson and Goodearl (see \cite{BG}) :
\begin{theorem}\label{Benson}
Let $\Ga$ be a group, and $H$ a finite index subgroup. Suppose
that $M$ is a $\Ga$-module which is projective over $H$ and flat
over $\Ga$. Then $M$ is projective over $\Ga$.
\end{theorem}
\begin{remark}
Notice that in order to apply the theorem, it is not necessary
that $(\Ga,H)$ satisfies Moore's condition. However, we shall use
the theorem in those cases only.\end{remark}

\begin{proof} As can easily be seen, we need to consider only the
special cases in which $\phi$ is one to one and in which $\phi$ is onto.
So suppose that $\phi$ is one to one and let $M$ be a $\Ga$-module
which is projective over $H$. Consider the induced module
$\tilde{M}=Ind^{\Ga'}_{\Ga}M$. This module is a $\Ga'$-module
which is projective over $H'$, so by assumption it is projective
over $\Ga'$. But when restricting to $\Ga$, $M$ is a direct
summand of $\tilde{M}$, and thus $M$ is $\Ga$-projective. The case
where $\phi$ is onto is more subtle. We shall prove that if $M$ is
a $\Ga$-module which is $H$-projective, then it is $\Ga$-flat. Then
we use Theorem \ref{Benson} to complete the proof. So let $M$ be such a module. We would
like to show that for every $\Ga$-module $N$ we have that
$Tor^{\Ga}_n(M,N)=0$ for every $n>0$. Consider first the case in
which $N$ is free as an abelian group. In this case it is easy to
see that $M\otimes_{\Z}N$ is projective as an $H$-module, where
$H$ acts diagonally. This is true due to the fact that we can
reduce easily to the case in which $M$ is a free $\Z H$-module,
and in that case it is true that as a $\Z H$-modules, $M\otimes_{\Z}N\cong
M\otimes_{\Z}N_{tr}$, which is a free $H$-module, where $N_{tr}$
is $N$ with the trivial $H$-action. Let us denote the kernel of
$\phi$ by $K$. The torsion groups $Tor^{\Ga}_n(M,N)$ are the
homology groups of the complex $P^*\otimes_{\Ga}N\rightarrow
M\otimes_{\Ga}N$, where $P^*\rightarrow M$ is a projective
resolution of $M$ over $\Ga$. Consider $P^*\otimes_{\Z} N\rightarrow
M\otimes_{\Z}N$ as a complex of $K$-modules. As such, it is a
resolution of a projective module by projective modules, and thus
it splits. It follows that after applying the functor $(-)_K$ of taking $K$-coinvariants, the
complex stays exact, and thus $P\otimes_K N\rightarrow M\otimes_K
N$ is still a resolution. $M\otimes_{\Z} N$ was projective over
$H$, and therefore $(M\otimes_{\Z}N)_K = M\otimes_K N$ is
projective over $H/K=H'$. But by assumption, $M\otimes_K N$ is
also projective over $\Ga'$. So by applying the functor
$(-)_{\Ga'}$, we get that the complex $P\otimes_{\Ga}N\rightarrow
M\otimes_{\Ga}N$ is exact, as required. In case $N$ is not free as
an abelian group, we proceed in the following way:
take a resolution of $N$ by $\Ga$-modules which
are free as abelian groups, $0\rightarrow Y\rightarrow
X\rightarrow N\rightarrow 0$, and use the long exact sequence in
homology in order to deduce that $Tor^{\Ga}_n(M,N)=0$ for $n>1$.
Consider a short exact sequence $0\rightarrow Q\rightarrow P\rightarrow M\rightarrow 0$ of $\Ga$-modules such that $P$ is projective over $\Ga$.
Since $M$ is projective over $H$, the same holds for $Q$, since the sequence splits over $H$. Now if $N$ is \emph{any} $\Ga$-module, we have that for every $n>0$ $Tor^{\Ga}_n(Q,N)=Tor^{\Ga}_{n+1}(M,N)=0$. The module $Q$ is thus $\Ga$ projective, and thus $M$ is projective over $H$ and has a projective resolution of length 1 over $\Ga$. It follows from Corollary \ref{finiteness} and Remark \ref{finiteness remark} that $M$ is projective over $\Ga$, as desired.\end{proof}
\end{section}

\begin{section}{Kropholler's hierarchy and the operator $\h$}\label{Kropholler}

Let $\x$ be any class of groups. Following Kropholler (see
\cite{K2}) we define the class $\h_1\x$ to be the class of all
groups $G$ which satisfies the following condition: $G$ acts on a
contractible finite dimensional CW-complex $X$ via a cellular
action, such that the setwise and the pointwise stabilizers of the
cells coincide, and the stabilizers of the cells are subgroups of
$G$ which lies inside $\x$.

Notice that actually $\x\subseteq \h_1\x$ since any group $G$ acts
on the one point CW-complex trivially, and the stabilizer of the
one point is $G$ itself. We now define by transfinite induction
the class $\h\x$. Define $\h_0\x = \x$,
$\h_{\alpha+1}\x=\h_1(\h_{\alpha}\x)$, and
$\h_{\lambda}\x=\bigcup_{\gamma<\lambda}\h_{\gamma}\x$ for a limit
ordinal $\lambda$. We define $\h\x$ to be the union of
$\h_{\alpha}\x$ over all ordinals. Notice that the class $\h\x$ is
closed under the operator $\h_1$ in the sense that
$\h_1(\h\x)=\h\x$. We denote by $\f$ the class of all finite
groups. Thus the class $\h\f$ is defined. For any class of groups
$\x$ we define $L\x$ to be the class of all groups $G$ such that
every finitely generated subgroup of $G$ is in $\x$ (that is- $G$
is locally inside $\x$). In particular, the class $\lhf$ is
defined. For a thorough investigation of the class $\lhf$, the
reader is advised to look at papers by Kropholler et al.
\cite{K1}, \cite{K2}, \cite{K3} and \cite{K4}.

In \cite{ACGK} Aljadeff et al. have proved that if $\Ga$ is in
$\lhf$, and $H$ is a finite index subgroup of $\Ga$ such that
$(\Ga,H)$ satisfies's Moore's condition, then Moore's conjecture
is true for finitely generated modules, that is- if $M$ is a
finitely generated $\Ga$-module which is projective over $H$, then
it is also projective over $\Ga$. We would like to prove that this
is actually true without the assumption that the module $M$ is
finitely generated.

The following was proved by the first author in \cite{Al2}:
\begin{proposition}\label{localness}

Let $\Ga$ be a group, and $H$ a finite index subgroup. Suppose
that for every finitely generated subgroup $K$ of $\Ga$ it is true
that $(K,K \cap H)$ satisfies Moore's conjecture. Then $(\Ga,H)$
satisfies Moore's conjecture.

\end{proposition}

\begin{proof}

Let $M$ be a $\Ga$-module which is $H$-projective. Then by the
assumption, $M$ is projective (and hence flat) over any finitely
generated subgroup $K$ of $\Ga$. This implies that $M$ is actually
flat over $\Ga$. Using Theorem \ref{Benson} above, the result
follows.\end{proof}

We can now prove theorem \ref{LHF thm} stated in the introduction:
\begin{proposition}\label{Kropholler1}
Let $\Ga\in \lhf$, and let $H$ be a finite index subgroup of $\Ga$
such that $(\Ga,H)$ satisfies Moore's condition. Then Moore's
conjecture is true for $(\Ga,H)$.
\end{proposition}

\begin{proof}
We first prove the theorem for $\Ga\in \h\f$. We use induction on
the first $\alpha$ for which $\Ga\in \h_{\alpha}\f$. The case
$\alpha=0$ is considered in section \ref{general}. The case where
$\alpha$ is a limit ordinal cannot happen as in that case
$\Ga\in\h_{\gamma}\f$ for some $\gamma<\alpha$. Suppose that
$\alpha=\gamma+1$. Then $\Ga$ acts on a contractible CW-complex
$X$ of dimension $n$ such that the stabilizers of the cells are
subgroups $K_i$ of $\Ga$ which lies inside $\h_{\gamma}\f$.
Consider the cellular chain complex $C$ of $X$, $$C=0\rightarrow
C_n\rightarrow C_{n-1}\rightarrow\ldots\rightarrow C_0\rightarrow
\Z\rightarrow 0.$$ It is finite dimensional since $X$ is finite
dimensional, and it is acyclic since $X$ is contractible. The
action of $\Ga$ on $X$ induces an action of $\Ga$ on $C$. As such,
each $C_j$ decomposes as a direct sum of permutation modules of
the form $\Z\Ga/K_i$, with one direct summand corresponds to each
orbit of the action of $\Ga$ on the $j$-cells. Let now $M$ be a
$\Ga$-module which is projective over $H$. By restriction, $M$ is
a $K_i$ module which is projective over $K_i\cap H$ for every $i$.
Since $(\Ga,H)$ satisfies Moore's condition, the same is true for
the pair $(K_i,K_i\cap H)$. But $K_i\in \h_{\gamma}\f$ so by
induction we can assume that actually $M$ is projective over $K_i$
for every $i$. By taking the tensor product of $C$ with $M$ over $\Z$, we
get the exact sequence $$C\otimes M =0\rightarrow C_n\otimes
M\rightarrow\ldots\rightarrow C_0\otimes M\rightarrow M\rightarrow
0.$$ The module $C_j\otimes M$ decomposes as the direct sum of module of
the form $\Z\Ga/K_i\otimes M$, which is isomorphic to
$Ind_{K_i}^{\Ga}M$ via the isomorphism $gK_i\otimes m\mapsto
g\otimes g^{-1}m$. Since $M$ is $K_i$ projective, it follows that
$Ind_{K_i}^{\Ga}M$ is $\Ga$-projective, and thus all the modules in
the complex $C\otimes M$ (beside $M$ perhaps) are projective. But
this means that $M$ has a projective resolution of finite length.
Using corollary \ref{finiteness} and remark \ref{finiteness remark}, we conclude that $M$ is
projective as well. In case we only know that the group $\Ga$ is in
$\lhf$, the argument we used shows that a $\Ga$-module $M$ which
is projective over $H$ is also projective over every finitely
generated subgroup $K$ of $\Ga$. Using proposition \ref{localness}
we conclude that $M$ is projective over $\Ga$ as well.\end{proof}

The argument used in the proof above can be used to
prove a stronger result. Notice first that it was not
really necessary that the complex $X$ would be contractible.
It would be enough if it were acyclic, since we have
only used the acyclicity of the complex $C$. In order to state the
generalization we need to make some definitions. Let $\y$ be a
class of pairs of a group and a finite index subgroup. This means that
elements of $\y$ are pairs of the form $(\Ga,H)$ for some group $\Ga$
and a finite index subgroup $H$ of $\Ga$. We define $\h_1\y$ to be
the class of all pairs $(\Ga,H)$ of a group $\Ga$ and a finite
index subgroup $H$, such that $\Ga$ acts cellularly on an acyclic
finite dimensional CW-complex $X$ such that the pointwise
stabilizer and the setwise stabilizer of all the cells coincide,
and such that for every stabilizer subgroup $K$ of one of the
cells, it holds that $(K,K\cap H)$ is in $\y$. We define $\h\y$
and $\lh\y$ in the same way we have defined before $\h\x$ and $\lh\x$.

We have the following:

\begin{proposition}\label{class-close}

Let $\y$ be the class of all pairs $(\Ga,H)$ for which Moore's
conjecture holds. Then $\y=\l\h\y$.

\end{proposition}

\begin{proof}

Exactly the same as the proof of \ref{Kropholler1}

\end{proof}

\end{section}

\begin{section}{cases where $\beta$ is not nilpotent}\label{main section}

%Let $(\Ga,H)$ be a pair of a group and a normal subgroup of prime
%index which satisfies Moore's condition.
The goal of this section is to prove Theorem \ref{main1}.
This yields examples for pairs of groups $(\Ga,H)$ which satisfies
Moore's conjecture, even though the Bockstein $\beta$ is not nilpotent.
Our proof uses a construction of Baumslag Dyer and Heller from \cite{BDH}.

Recall that a group $G$ is said to be \emph{acyclic} if $H_n(G,\Z)=0$
for every $n>0$. Using the universal coefficients theorem,
one can see that this determines that $H^n(G,\Z)=0$ for
every $n>0$. Recall first a definition and two propositions from
\cite{BDH}

\begin{definition}

A supergroup $M$ of a group $B$ is called a \emph{mitosis} of $B$
if there exist elements $s$ and $d$ of $M$ such that \\1. $M=\langle
B,s,d\rangle$ \\2. $b^d=bb^s$ for all $b\in B$, and \\3. $[b',b^s]=1$
for all $b,b'\in B$. \\A group $M$ is called \emph{mitotic} if it
contains a mitosis of every one of its finitely generated
subgroups.

\end{definition}

\begin{proposition}

Mitotic groups are acyclic.

\end{proposition}

\begin{proposition}\label{mitoticity}

Let $G$ be any group. Then $G$ can be embedded in a mitotic group
$A$ (and thus each group is embeddable in an acyclic group). If
$G$ is in $\lhf$ then we can choose $A$ to be in $\lhf$ as
well.

\end{proposition}

\begin{proof}

The idea is that we begin with a group $G$, embed it into a
mitosis of $G$, $m(G)$, embed $m(G)$ into $m^2(G)$ which is a
mitosis of $m(G)$, and so on, and we take $A$ to be the union. Clearly $A$ is mitotic.
The group $m(G)$ is defined in the following way: let $D=G\times G$, let
$E=\langle D,t;t^{-1}(g,1)t = (g,g), g\in G\rangle$ and let $m(G)
= \langle E,u;u^{-1}(g,1)u=(1,g),g\in G\rangle$. It is easy to see
that indeed $m(G)$ is a mitosis of $G$. If $G$ is in $\lhf$, then
the same is true for $D$, $E$ and $m(G)$, since $\lhf$ is closed
under taking direct products, HNN extensions, and countable
unions.\end{proof}

We proceed now to prove Theorem \ref{main1}. We state and prove a somewhat
more general version using the above terminology.

\begin{proposition}\label{list-pro}

Let $(\Ga_0,H_0)$ be a pair of a group and a finite index normal
subgroup. The group $\Ga_0$ can be embedded in a group $\Ga$ which
has a finite index normal subgroup $H$ such that the following
conditions hold: \\1. $H\cap \Ga_0=H_0$. \\2. the inclusion
$\Ga_0\rightarrow \Ga$ induces an isomorphism $\Ga_0/H_0\cong
\Ga/H$. \\3. The pair $(\Ga,H)$ satisfies Moore's condition if and
only if $(\Ga_0,H_0)$ does. \\4. The pair $(\Ga,H)$ satisfies
Moore's conjecture if and only if $(\Ga_0,H_0)$ does. \\5. $\Ga$ is
in $\lhf$ if and only if $\Ga_0$ is. \\6. The group $H$ is mitotic
(and thus acyclic).

\end{proposition}

In order to prove the proposition, we first prove the following
lemma:

\begin{lemma}\label{main lemma}

The group $\Ga_0$ can be embedded in a group $\Ga_1$ which has a
finite index subgroup $H_1$ such that conditions 1-5 of the
proposition hold, and such $H_1$ contains a mitosis of $H_0$.\end{lemma}

\begin{proof}

Consider an embedding of $H_0$ in a mitotic group $A$. By lemma
\ref{mitoticity}, we may assume that if $\Ga$ is in $\lhf$, then
so does $A$. Define $\Ga_1=\Ga_0*_{H_0}A$. Then $\Ga_0$ is a
subgroup of $\Ga_1$. There exists an epimorphism
$\phi:\Ga_1\rightarrow \Ga_0/H_0$ which sends $A$ to the trivial
element, and which maps $\Ga_0$ onto $\Ga_0/H_0$ canonically.
Denote $ker(\phi)$ by $H_1$. Note that $H_1$
contains $A$ which is a mitosis of $H_0$. It is easy to see that
$(\Ga_1,H_1)$ satisfies conditions 1,2, and 5 of the proposition.
We now prove that the pair $(\Ga_1,H_1)$ also satisfies conditions 3 and 4.
If $(\Ga_0,H_0)$ violates Moore's condition,
then there exists an element $x\in \Ga_0$ such that $\langle x
\rangle\cap H_0 =1$, and so the same $x$ violates Moore's
condition for $\Ga_1$. In the other direction, if there exist an
$x\in \Ga_1$ such that $\langle x \rangle \cap H_1 = 1$, then $x$
is of finite order, and thus is conjugate to an element of finite
order in $A$ or in $\Ga_0$. As $x\not\in H_1$, we conclude that
$x$ is conjugate to an element $y\in\Ga_0$, which violates Moore's
condition for $(\Ga_0,H_0)$. To prove 4, notice that if
$(\Ga_1,H_1)$ satisfies Moore's conjecture, then also
$(\Ga_0,H_0)$ satisfies Moore's conjecture by \ref{closure}. If
$(\Ga_0,H_0)$ satisfies Moore's conjecture, we proceed as follows:
$\Ga_1$ is the amalgamated free product of $A$ and $\Ga_0$ over
$H_0$, and thus $\Ga_1$ acts on a tree (which is in particular a
CW-complex) such that the points of the tree has two orbits, one
with stabilizer $A$ and the other with stabilizer $\Ga_0$, and the
edges have one orbit, with stabilizer $H_0$. Using now proposition
\ref{class-close}, the result follows.\end{proof}%
\emph{proof {of proposition \ref{list-pro}}} Using the lemma above, one can
form a chain of embeddings
$\Ga_0\injects\Ga_1\injects\Ga_2\injects\cdots$. We define $\Ga$
to be the union of the groups $\Ga_i$, and $H$ to be the union of
the subgroups $H_i$. It is easy to see that $(\Ga,H)$ satisfies
conditions 1 and 2. Condition 3 is satisfied by an argument similar to
the one used in Lemma \ref{main lemma}. Condition 4 is proved in one direction
using lemma \ref{closure} again. On the other direction we proceed
as follows: Suppose that Moore's conjecture is true for $(\Ga_0,H_0)$
%(for example- take $(\Z,p\Z)$, where $p$ is any prime number).
Then it is true also for every $(\Ga_i,H_i)$. As
every finitely generated subgroup of $\Ga$ is contained in some
$\Ga_i$, using proposition \ref{closure} and \ref{localness}, the
result follows. Condition 5 is true since Kropholler's hierarchy
is closed under countable unions. The only thing left to prove is
that $H$ is mitotic. This follows easily from the fact that
$H_{i+1}$ contains a mitosis of $H_i$.\qedsymbol

Notice in particular that if $H$ is acyclic, then
$H^*(\Ga/H,\Z)\cong H^*(\Ga,\Z)$ by inflation. This can easily be
seen by considering the LHS spectral sequence in cohomology which
corresponds the the group extension $$1\rightarrow H\rightarrow
\Ga\rightarrow \Ga/H\rightarrow 1.$$ As a result, take a pair
$(\Ga_0,H_0)$ such that $\Ga/H$ is cyclic of prime order, and such
that Moore's conjecture is true for $(\Ga_0,H_0)$. Embed the pair
in a pair $(\Ga,H)$ such that $H$ is acyclic, using the
proposition. Then $(\Ga,H)$ is a pair for which Moore's conjecture
is true even though the Bockstein is not nilpotent. If one takes
$\Ga_0$ to be a group in $\lhf$ then $\Ga$ would be in $\lhf$ as
well. If $\Ga_0\not\in\lhf$, (for example, take $\Ga_0$ to be
Thompson's group $F$), then the same holds for $\Ga$, since $\lhf$
is subgroup closed. In conclusion, we have the
following corollary:

\begin{corollary}

Any pair of groups $(\Ga_0,H_0)$ such that $H_0$ is a normal subgroup of finite
index in $\Ga_0$ can be embedded in a pair of groups $(\Ga,H)$
such that $\Ga\in\lhf$ if and only if $\Ga_0\in\lhf$, Moore's
conjecture is true for $(\Ga_0,H_0)$ if and only if it is true for
$(\Ga,H)$, and $H$ is acyclic. As a result, there are examples,
inside and outside $\lhf$, of pairs of groups $(\Ga,H)$ such that
$\Ga/H$ is cyclic of prime index, $(\Ga,H)$ satisfies Moore's
condition, and the Bockstein element is not nilpotent. In
particular, it is enough to prove Moore's conjecture for pairs of
groups $(\Ga,H)$ such that $\Ga/H$ is cyclic of prime order, $H$
is acyclic, and $(\Ga,H)$ satisfies Moore's
condition.

\end{corollary}

\end{section}


\begin{thebibliography}{99}

\bibitem[1]{ACGK} E. Aljadeff, J. Cornick, Y. Ginosar, P. Kropholler,
On a Conjecture of Moore, J. Pure Appl. Algebra 110,  (1996),
109–112.

\bibitem[2]{Al1} E. Aljadeff, On cohomology rings of infinite groups,
J. Pure Appl. Algebra, 208 (2007), 1099-1102

\bibitem[3]{Al2} E. Aljadeff, Profinite groups, profinite completions
and a conjecture of Moore, Advances in Mathematics, 201, 1 (2006),
63-76, arXiv:math/0405201v1.

\bibitem[4]{BDH} G. Baumslag, E. Dyer and A. Heller-
The topology of discrete groups, J. Pure Appl. Algebra,  16 (1980)
1-47.

\bibitem[5]{BG} D. Benson, K. Goodearl, Periodic flat
modules, and flat modules for finite groups, Pacific J. Math. 196
(2000), 45-67.

\bibitem[6]{Carlson} J. F. Carlson, Cohomology and induction from elementary Abelian subgroups,
The Quarterly Journal of Mathematics 2000 51(2):169-181; doi:10.1093/qjmath/51.2.169

\bibitem[7]{Ch} L.G. Chouinard, Projectivity and relative projectivity over group rings, J. Pure
Appl. Algebra 7 (1976), 287–302.

\bibitem[8]{Evens} L. Evens, The cohomology of groups,
Oxford mathematical monographs, Oxford University Press, (1991)

\bibitem[9]{K1} J. Cornick and P. H. Kropholler, Homological
finiteness conditions for modules over group algebras, J. London
Math. Soc. 58 (1998), 49–62.

\bibitem[10]{K2} P. H. Kropholler, On groups of type FP-infinity,
J. Pure Appl. Algebra 90 (1993), 55–67.

\bibitem[11]{K3} P. H. Kropholler, Hierarchical decompositions,
generalized Tate cohomology, and groups of type FP-infinity, (eds.
A. Duncan, N. Gilbert, and J. Howie) Proceedings of the Edinburgh
Conference on Geometric Group Theory, (1993) (Cambridge U. P.
1994).

\bibitem[12]{K4} P. H. Kropholler and G. Mislin, 'Group actions
on finite dimensional spaces with finite stabilizers', Comment.
Math. Helv. 73 (1998), 122–136.

\bibitem[13]{Maclane} S. Mac lane, Homology, Springer Verlag (1967)

\bibitem[14]{cong} L. Ribes and P. Zalesskii, Profinite Groups, Springer, Berlin (2000).

\bibitem[15]{Swan} R.G.Swan,  Groups of cohomological dimension one, J. Algebra 12 (1969), 585-610.
\bibitem[16]{KT} D.M. Kan and W.P. Thurston, Every connected space has the homology of a
K($\pi$, 1), Topology 15 (1976), No. 3, 253–258.

\bibitem[17]{QV} D. Quillen and B. B. Venkov, Cohomology of finite groups and elementary
abelian subgroups, Topology 11 (1972), 317–318.

\end{thebibliography}
\end{document}